\documentclass[10pt,a4paper]{amsart}

\usepackage{hyperref}
\usepackage{amsmath,amssymb,amsfonts}
\usepackage{mathtools}

\usepackage{stmaryrd}
\usepackage{mathrsfs}

\usepackage{todonotes}

\def\equationautorefname~#1\null{Equation~(#1)\null}

\usepackage[marginpar=2cm,ignoremp,margin=3cm]{geometry}

\usepackage{mathdots}
\usepackage{etoolbox}
\usepackage{shuffle}
\usepackage{thmtools}

\makeatletter
\def\@endtheorem{\endtrivlist}
\makeatother

\declaretheorem[
style=plain,
name=Theorem,
numbered=yes,
refname={Theorem,Theorems},
Refname={Theorem,Theorems}
]{Thm}
\declaretheorem[
style=plain,
name=Proposition,
numberlike=Thm,
refname={Proposition,Propositions},
Refname={Proposition,Propositions}
]{Prop}
\declaretheorem[
style=definition,
name=Remark,
numberlike=Thm,
refname={Remark,Remarks},
Refname={Remark,Remarks}
]{Rem}
\declaretheorem[
style=plain,
name=Conjecture,
numberlike=Thm,
refname={Conjecture,Conjectures},
Refname={Conjecture,Conjectures}
]{Conj}

\declaretheorem[
style=definition,
name=Example,
numberlike=Thm,
refname={Example,Examples},
Refname={Example,Examples}
]{Eg}
\declaretheorem[
style=plain,
name=Corollary,
numberlike=Thm,
refname={Corollary,Corollaries},
Refname={Corollary,Corollaries}
]{Cor}
\declaretheorem[
style=plain,
name=Lemma,
numberlike=Thm,
refname={Lemma,Lemmas},
Refname={Lemma,Lemmas }
]{Lem}

\DeclareMathOperator{\sech}{sech}

\newcommand{\poch}[2]{%
	{\left\{ #1 \right\}_{#2}}
}

\newcommand{\sgnarg}[2]{
	\left( \genfrac{}{}{0pt}{0}{#1}{#2} \right)
}

\newmuskip\pFqmuskip
\newcommand*\pFq[6][8]{%
	\begingroup 
	\pFqmuskip=#1mu\relax
	\mathchardef\normalcomma=\mathcode`,
	\mathcode`\,=\string"8000
	\begingroup\lccode`\~=`\,
	\lowercase{\endgroup\let~}\pFqcomma
	{}_{#2}F_{#3}{\left[\genfrac..{0pt}{}{#4}{#5};#6\right]}%
	\endgroup
}
\newcommand{\pFqcomma}{{\normalcomma}\mskip\pFqmuskip}

\let\eps\varepsilon

\let\overlineO\overline
\renewcommand{\overline}[1]{\overlineO{\mathclap{\phantom{I}}#1}}

\begin{document}
	
	\title[Differential operators, and a depth reduction for \protect\( \zeta(1, \ldots, 1, \overline{2m}) \protect\)]{Differential operators and a depth reduction for \\ the alternating multiple zeta values \( \zeta(1, \ldots, 1, \overline{2m}) \)}
	
	\author[Au]{Kam Cheong Au}
	\address{Rheinische Friedrich-Wilhelms-Universität Bonn \\ Mathematical Institute \\ 53115 Bonn, Germany} 
	\email{s6kmauuu@uni-bonn.de}
	
	\author[Charlton]{Steven Charlton}
	\address{Max Planck Institute for Mathematics \\ Vivatsgasse 7 \\ Bonn 53111 \\ Germany}
	\email{charlton@mpim-bonn.mpg.de}
	
	\author[Hoffman]{Michael E. Hoffman}
	\address{Department of Mathematics \\ U.S. Naval Academy \\ Annapolis \\ MD 21402 \\ USA}
	\email{meh@usna.edu}
	
	\makeatletter
	\@namedef{subjclassname@2020}{%
		\textup{2020} Mathematics Subject Classification}
	\makeatother
	
	

	\date{28 December 2023}
	
	\begin{abstract}
		We establish an identity amongst certain differential operators applied to a formal power-series.  As a corollary we obtain an explicit depth reduction result for alternating MZV's of the form \( \zeta(1,\ldots,1,\overline{2m}) \), which resolves a conjecture posed earlier by the third author.
	\end{abstract}	
	
	\maketitle
	
	\section{Introduction}
	
	The alternating multiple zeta values (MZV's) with signs \( \eps_i \in \{ \pm 1 \} \) and arguments \( k_i \in \mathbb{Z}_{>0} \) are defined by
	\[
	\zeta\sgnarg{\eps_1,\ldots,\eps_d}{k_1,\ldots,k_d} \coloneqq \sum_{0 < n_1 < \cdots < n_d} \frac{\eps_1^{n_1} \cdots \eps_d^{n_d}}{n_1^{k_1} \cdots n_d^{k_d}} \,.
	\]
	The series converges for \( (\eps_d,k_d) \neq (1,1) \).  Frequently the notation is simplified by writing the MZV with a single row of arguments, and decorating \( k_i \) with a bar, i.e. \( \overline{k_i} \), exactly when the corresponding \( \eps_i = -1 \).  For example
	\[
	\zeta(\overline{k_1}, k_2, \overline{k_3}, \overline{k_4}, k_5) = \zeta\sgnarg{-1, 1, -1, -1, 1}{k_1,k_2,k_3,k_4,k_5} \,.
	\]
	In this article we will be interested in alternating MZV's of the form \( \zeta(\{1\}^k, \overline{m}) \), where \( \{1\}^k \) denotes the string \( 1, \ldots, 1 \) with \( k \) repetitions. \medskip
	
	Recall that the Euler polynomials \( E_n(x) \) are defined by
	\begin{align*}
	\sum_{j=0}^\infty E_j(x) \frac{t^j}{j!} & \coloneqq \frac{2 e^{t x}}{1 + e^t} \,. \\
	&= 1 + \Big( x - \frac{1}{2} \Big) t  + \big( x^2 - x \big) \frac{t^2}{2!} + \Big( x^3 - \frac{3}{2} x^2 + \frac{1}{4} \Big) \frac{t^3}{3!} + \cdots \,.
	\end{align*}
	For a polynomial or formal power-series \( F(x) \), denote by \( [x^a] F(x) \) the coefficient of \( x^a \) in \( F(x) \).   \medskip
	
	The following conjecture was posed by the third author, during a talk \cite{Hoffman-him} given in January 2018 at the \emph{Periods in Number Theory, Algebraic Geometry and Physics} trimester, at the Hausdorff Research Institute for Mathematics, in Bonn. 
	\begin{Conj}[Hoffman, 2018, \cite{Hoffman-him}]\label{conj:z111evenbar}
		For integers \( k, m \geq 1 \), the alternating MZV \( \zeta(\{1\}^k, \overline{2m}) \) can be written as
		\[
		\zeta(\{1\}^k, \overline{2m}) = 
		\begin{aligned}[t] 
		& P(\zeta(2), \, \zeta(3), \zeta(5), \ldots, \zeta(2i+1),\ldots) \\
		& - \sum_{n=m+1}^{m+\lceil k/2 \rceil} \, [z^{2m-2}] E_{2n-3}(z) \cdot \zeta(\{1\}^{k+1-2(n-m)}, \overline{2n-1}) \,, \end{aligned}
		\]
		where \( P \in \mathbb{Q}[x_2 \,,\, x_3,x_5,\ldots,x_{2i+1},\ldots] \) is some (unspecified) rational polynomial.
	\end{Conj}
	
	This conjecture originated as an intermediate step in a attempt to reduce a certain integral to classical multiple zeta values; in \cite{Hoffman-integral} these integrals were directly reduced to classical multiple zeta values, bypassing the conjecture completely.  The conjecture itself, however, is not a consequence of this alternative proof; the proof only implies the (expected-to-be) irreducible terms in certain linear combinations of alternating MZV's vanish.  One is therefore still interested in giving a proof of this conjecture, which will be the focus of this article.
	
	\begin{Rem}
		By the depth-parity theorem for alternating multiple zeta values (or more generally multiple polylogarithms cf. \cite{Goncharov,Panzer}), one knows that an alternating MZV of weight \( w \), depth \( d \) with \( w \not\equiv d \pmod{2} \) reduces to lower depth.  Since \( \zeta(\{1\}^k, \overline{2m}) \) has weight \( w = k + 2m \), and depth \( d = k + 1 \), the depth-parity theorem implies \( \zeta(\{1\}^k, \overline{2m}) \) always reduces to lower depth, i.e. depth \( \leq k \).  Our proof of \autoref{conj:z111evenbar} therefore gives an explicit depth-parity result for \( \zeta(\{1\}^k, \overline{2m}) \).
		
		Unfortunately, both of these parity theorems \cite{Goncharov,Panzer} are hard to apply directly to our situation.  
		In the case of Panzer \cite{Panzer}, only the depth 2 and 3 cases of the parity theorem are stated explicitly.  The higher depth cases are constructed recursively via an algorithm presented therein.
		
		Whereas in the case of Goncharov \cite{Goncharov} we would actually obtain a reduction for the regularised expression \( \zeta(\{1\}^k, \overline{2m}) + (-1)^k \zeta(\overline{2m},\{1\}^k) \).  One must then somehow use the stuffle-product to write 
		\[
		\zeta(\overline{2m},\{1\}^k) \equiv (-1)^k \zeta(\{1\}^k,\overline{2m}) \pmod{\text{products, depth $\leq k$}} \,;
		\]
		This ends up leading to the question of whether
		\[
		\sum\nolimits_{i=0}^k \zeta(\{1\}^i,2,\{1\}^{k-i},\overline{m}) + \zeta(\{1\}^{k+1}, \overline{m},\overline{1}) \overset{?}{\in} \mathbb{Q}[\zeta(2), \ldots, \zeta(1,\ldots,1,\overline{2m+1}),\ldots] \,,
		\]
		which seems not straightforward to answer.  (Our parity theorem implies this answer is affirmative, however.)
	\end{Rem}
	
	We give a proof of \autoref{conj:z111evenbar} by explicitly showing a certain identity amongst derivatives of \( {}_2F_1 \) functions, arising from the application of formal differential operators.  On one hand this identity gives a generating series for the depth \( {>}1 \) part of the above conjecture, whilst on the other hand it is expressed via combinations which reduce to Gamma functions and derivatives thereof.  This will give an explicit evaluation for the depth \( {>}1 \) of \autoref{conj:z111evenbar} via polynomials in single zeta-values, and thus refining and proving the conjecture. \medskip 
	
	Recall the Laplace transform of a (sufficiently nice) function \( f(y) \) is defined by
	\[
	\mathscr{L}_y\{f\}(s) \coloneqq \int_0^\infty f(y) \exp(-s y) \mathrm{d}y \,.
	\]
	(The function \( f \) may also depend on other variables).  One key property of the Laplace transform is that it converts exponential generating series into ordinary generating series, as follows.  A direct computation shows
	\[
	\mathscr{L}_y\{y^n\}(s) = \frac{n!}{s^{n+1}} \,,
	\]
	so when one takes the Laplace transform of the exponential generating series of a sequence \( \lambda_n \) one obtains (convergences issues aside),
	\[
	\mathscr{L}_y\bigg\{ \sum_{n=0}^\infty \frac{\lambda_n}{n!} y^n \bigg\}(s) = \frac{1}{s} \cdot \sum_{n=0}^\infty \frac{\lambda_n}{s^n} \,.
	\]
	This is nothing other than the ordinary generating series the sequence \( \lambda_n \), evaluated at \( s^{-1} \), up to a prefactor.  We will extend the Laplace transform \( \mathscr{L}_y\{\bullet\}(s)  \) coefficient-wise to a formal power-series in a new variable \( x \), whose coefficients are \emph{polynomials} in \( y \).
	
	More precisely, consider a formal power-series \( f(y,\partial_a,\partial_b; x) \in \mathbb{Q}[y, \partial_a, \partial_b]\llbracket x \rrbracket \) in the variable \( x \), whose coefficients are polynomials in \( y \) and the differentials \( \partial_a, \partial_b \).  (For emphasis, we separate the power-series variable from the polynomials variables notationally.)  Write
	\[
	f(y,\partial_a,\partial_b; x) = \sum_{n=0}^\infty \lambda_n(y, \partial_a, \partial_b) x^n \,,
	\]
	then define the formal Laplace transform (with respect to \( y \)) as follows:
	\[
	\mathscr{L}_y\{f\}(s) = \sum_{n=0}^\infty \mathscr{L}_y\big\{ \lambda_n(y, \partial_a, \partial_b) \big\}(s) x^n \in \mathbb{Q}[s^{-1}, \partial_a, \partial_b]\llbracket x \rrbracket \,.
	\]
	Now introduce the following differential operators (the functions inside \( \mathscr{L}_y\{\bullet\}(s) \) are elements of \( \mathbb{Q}[y,\partial_a,\partial_b]\llbracket x \rrbracket \), as is straightforward to check):
	\begin{align}
	\label{eqn:d1}	\mathcal{D}_1 & \coloneqq s \cdot \mathscr{L}_y\bigg\{
	\begin{aligned}[t]
	& \frac{y}{2}-\frac{1}{2} (\partial_a-\partial_b) \tanh \Big(\frac{x y}{2}\Big)-\frac{1}{4} x y (\partial_a-\partial_b) \sech^2\Big(\frac{x y}{2}\Big) \\
	& +\frac{1}{2} y \sech\Big(\frac{x y}{2}\Big) \sinh \Big(\frac{1}{2} x (\partial_a-\partial_b)\Big) \sech\Big(\frac{1}{2} x (\partial_a-\partial_b+y)\Big) \\
	& -\frac{1}{x} \sinh \Big(\frac{x y}{2}\Big) \exp\Big(\frac{1}{2} x (\partial_a+\partial_b)\Big) \sech\Big(\frac{1}{2} x (\partial_a-\partial_b+y)\Big)  \bigg\}(s) \Big\rvert_{a=-s,b=s} \,,
	\end{aligned} \\
	\label{eqn:d2}	\mathcal{D}_2 & \coloneqq s \cdot \mathscr{L}_y\bigg\{ 
	\tanh\Big( \frac{x(y - \partial_b)}{2} \Big)\bigg\}(s) \Big\rvert_{b=s} \,, \\
	\label{eqn:d3}	\mathcal{D}_3 & \coloneqq s \cdot \mathscr{L}_y\bigg\{ -\frac{1}{4} x \Big( -2 + x y \tanh \Big(\frac{xy}{2}\Big)\Big) \sech^2\Big(\frac{x	y}{2}\Big) \bigg\}(s) \Big\rvert_{b=s} \,.
	\end{align}
	We understand the action of a differential operator \( \mathcal{E} = \sum_{n=0}^\infty \mu_n(s, \partial_a,\partial_b) x^n \in \mathbb{Q}[s^{-1},\partial_a, \partial_b]\llbracket x \rrbracket \) on some function \( g = g(a,b,x) \) to be the formal power-series
	\[
	\mathcal{E} \circ g = \sum_{n=0}^\infty ([x^n]\mathcal{E} \, \circ \, g) x^n \,,
	\]
	where \( [x^n]\mathcal{E} = \mu_n(s, \partial_a,\partial_b) \) is a polynomial in \( s,\partial_a, \partial_b \), so already has a well-defined action.  Likewise, when the operator contains a substitution \( \big\rvert_{a=\bullet,b=\bullet} \). \medskip
	
	\autoref{conj:z111evenbar} then follows from the following theorem, which is a special case of the main theorem of this article, \autoref{thm:main} below.  (Although \( \mathcal{D}_3 \) plays a trivial role in this special case, we retain it for consistency and comparison with the more general theorem.)
	\begin{Thm}[\autoref{cor:z11evbar} below]\label{thm:z11evbar}
		The following generating series identity holds as a formal power-series in \( x \),
		\begin{align*}
		& \sum_{k=0}^\infty \sum_{m=1}^\infty \Big\{ \zeta(\{1\}^k, \overline{2m}) + \sum_{n=m+1}^{m+\lceil k/2 \rceil} [z^{2m-2}] E_{2n-3}(z) \zeta(\{1\}^{k+1-2(n-m)}, \overline{2n-1})	\Big\} x^{k+1} s^{2m-1} \\
		& = \mathcal{D}_1 \circ \Big( \frac{\Gamma(1+a)\Gamma(1+b)}{\Gamma(a+b)}  \Big)  +  \mathcal{D}_2 \circ 1 
		+ \underbrace{\mathcal{D}_3 \circ 0}_{=0}  \,,
		\end{align*}
		where \( \Gamma \) is the Gamma function.
	\end{Thm}
	From \autoref{thm:z11evbar}, we can extract the coefficient of \( x^{k+1} \) for any \( k \), in order to find explicit results.  
	\begin{Eg}[Expression for \( \zeta(1, 1, \overline{2m}) \)]
		We extract the coefficient of \( x^{k+1} \), \( k = 2 \) in \autoref{thm:z11evbar}.  We note that
		\[
		[x^3] \mathcal{D}_1 = 
		\frac{\partial _a \partial _b}{4 s^2}
		-\frac{\partial _a \left(\partial _b\right){}^2}{8 s}
		+\frac{\partial_a}{4 s^3}-\frac{\left(\partial_b\right){}^2}{4s^2}
		+\frac{\left(\partial _b\right){}^3}{24 s} \Big\rvert_{a=-s,b=s} \,.
		\]
		We therefore compute (via the limit \( \eps \to 0 \), as \( a + b = 0 \) is a removable singularity), that
		\begin{align*}
		& \lim_{\eps\to0} \bigg( \frac{\partial _a \partial _b}{4 s^2}-\frac{\partial _a \left(\partial _b\right){}^2}{8 s}+\frac{\partial _a}{4 s^3}-\frac{\left(\partial
			_b\right){}^2}{4 s^2}+\frac{\left(\partial _b\right){}^3}{24 s} \bigg) \frac{\Gamma(1+a)\Gamma(1+b)}{\Gamma(a+b)} \Big\rvert_{a=-s+\eps,b=s} \\
		& = \begin{aligned}[t] 
		& \frac{\pi  \csc (\pi  s)}{4 s^2}+\frac{1}{24} \pi ^3 \csc (\pi  s)+\frac{\pi  \csc (\pi  s) (\psi(1-s)+\gamma )}{4 s} \\
		& -\frac{1}{4}
		\pi  \csc (\pi  s) (\psi(1-s)+\gamma ) (\psi(1+s)+\gamma )-\frac{\pi  \csc (\pi  s) (\psi(1+s)+\gamma )}{4 s} \,.
		\end{aligned}
		\end{align*}
		Here \( \psi(s) = \frac{\mathrm{d}}{\mathrm{d}s} \log\Gamma(s) \) is the digamma function, with Taylor series
		\[
		\psi(1+s) = \frac{\mathrm{d}}{\mathrm{d}s} \log\Gamma(1+s) = -\gamma + \sum_{n=2}^\infty  (-1)^n s^{n-1} \zeta(n) \,.
		\]
		We also have that \( [x^{k+1}] \mathcal{D}_2 \circ 1  = -\frac{E_{k+1}(0)}{s^{k+1}} \), so for \( k= 2 \) we obtain
		\(
		[x^{3}] \mathcal{D}_2 \circ 1 = -\frac{1}{4s^3} \,.
		\)	
		So overall, we have the identity
		\begin{align*}
		&  \sum_{m=1}^\infty \Big\{ \zeta(1,1, \overline{2m}) - \frac{2m-1}{2} \zeta(1, \overline{2m+1})	\Big\}  s^{2m-1}  \\[-0.5ex]
		& = \begin{aligned}[t]  
		&  -\frac{1}{4s^3} + \frac{\pi  \csc (\pi  s)}{4 s^2}+\frac{1}{24} \pi ^3 \csc (\pi  s)+\frac{\pi  \csc (\pi  s) (\psi (1-s)+\gamma )}{4 s} \\
		& -\frac{1}{4}
		\pi  \csc (\pi  s) (\psi(1-s)+\gamma ) (\psi(1+s)+\gamma )-\frac{\pi  \csc (\pi  s) (\psi(1+s)+\gamma )}{4 s} \,.
		\end{aligned}
		\end{align*}
		
		By comparing this with the explicit form conjectured in \cite{Hoffman-him}, namely for some \( r_m \overset{?}{\in} \mathbb{Q} \),
		\[
		\zeta(1,1,\overline{2m}) \overset{?}{=} \begin{aligned}[t]
		&  \frac{2m-1}{2} \zeta(1,\overline{2m+1}) + r_m \zeta(2m+2) - \frac{1}{2} \sum_{j=2}^{2m} \hat{\zeta}(j)\hat{\zeta}(2m+2-j) \\[-1ex]
		& + \frac{1}{6} \sum_{\substack{i+j+k = 2m+2 \\ i,j,k\geq2}} \hat{\zeta}(i)\hat{\zeta}(j)\hat{\zeta}(k)
		\,,
		\end{aligned}
		\]
		where \( \hat{\zeta}(i) = \zeta(i) \), if \( i \) odd, and \( = \zeta(\overline{\,i\,}) \), if \( i \) even, we can give an expression for the (heretofore conjecturally) rational coefficients \( r_m \).  Sum the above conjecture to a generating series, and subtract it from the proven identity above.  We are left with the result that \( \sum_m r_m \zeta(2m+2) t^{2m-1} \) is expressed via a certain trigonometric series.  Explicitly we find
		\begin{align*}
		\sum_{m=0}^\infty r_m \zeta(2m+2) t^{2m-1} 
		&= -\frac{5}{24 t^3}+\frac{1}{12} \pi ^3 \csc ^3(\pi 
		t)-\frac{1}{48} \pi ^3 \csc (\pi  t)+\frac{\pi ^2 \cot (\pi
			t) \csc (\pi  t)}{8 t}
		\\
		&= \frac{3}{32} \zeta (4) t+\frac{151}{192} \zeta (6)
		t^3+\frac{3287 \zeta (8) t^5}{1536}+\frac{10629 \zeta (10)
			t^7}{2560} + \cdots \,,
		\end{align*}
		which establishes the rationality of the \( r_m \).
	\end{Eg}
	
	The outline for the rest of the paper is as follows.  In \autoref{sec:state} we will state a strengthened result \autoref{thm:main}, and show this implies \autoref{thm:z11evbar} as a corollary.  We give the proof of \autoref{thm:main} in \autoref{sec:proof} using the theory of differential operators.\medskip

	\paragraph{\bf Acknowledgements:}  SC would like to thank the Max Planck Institute for Mathematics, in Bonn, for their support, hospitality and excellent working conditions.  SC and MEH are -- belatedly -- very grateful to the Hausdorff Research Institute for Mathematics, and to the organisers of the \emph{Periods in Number Theory, Algebraic Geometry and Physics} trimester program, for the lively discussion and collaboration environment, from which this paper eventually grew.
	
	\section{Statement of the main theorem, and its consequences}\label{sec:state}
	
	Recall the differential operators \( \mathcal{D}_1, \mathcal{D}_2, \mathcal{D}_3 \) from Equations \eqref{eqn:d1}--\eqref{eqn:d3}.  The main theorem of this article is as follows.
	\begin{Thm}\label{thm:main}
		For any (continuously differentiable) function \( f(a,b) \), the following identity holds when it is viewed as a formal power-series in \( x \),
		\begin{align*}
		& ( f(x, s) - f(0, s)) \\
		& + \Big( \frac{2s}{1 + \exp(x \partial_b)} \Big\rvert_{b=s} \Big) \circ \bigg( {-}\frac{ f(x, -b) + f(x, b)}{2b} + \frac{ f(0, -b) + f(0, b) }{2b} \bigg) \\[2ex]
		& = \begin{aligned}[t] 
		& \mathcal{D}_1 \circ \big(b f(a+b, a) + a f(a+b, b)\big) \\
		& + \mathcal{D}_2 \circ f(0, b) \quad + \quad \mathcal{D}_3 \circ \big( \partial_b ( f(0, b) - f(0, -b) ) \big) \,. \end{aligned}
		\end{align*}
	\end{Thm}
	
	The proof is given in \autoref{sec:proof}.  We first show how this theorem implies the above result about alternating multiple zeta values in \autoref{thm:z11evbar}. \medskip
	
	To start with, we note the following generating series for our alternating MZV's.
	\begin{Lem}
		The following alternating MZV's generating series holds
		\[
		\sum_{k,n>0} \zeta(\{1\}^{k-1}, \overline{n+1}) x^k s^{n} = 1 - \pFq{2}{1}{x,-s}{1-s}{-1}\,.
		\]
		
		\begin{proof}
			By direct calculations, the left hand side is equal to
			\begin{align*}
			\sum_{n=1}^\infty \sum_{m=1}^\infty x \cdot \prod_{i<m} \Big( 1 + \frac{x}{i} \Big) \cdot \frac{(-1)^m s^n}{m^{n+1}} 
			&= \sum_{m=1}^\infty \bigg( {-}\,\frac{\poch{x}{m} \poch{-s}{m}}{\poch{1-s}{m}} \bigg) \frac{(-1)^m}{m!}
			\\
			&= 1 - \pFq{2}{1}{x,-s}{1-s}{-1} \,,
			\end{align*}
			which proves the claim.
		\end{proof}
	\end{Lem}
	
	\begin{Lem}
		The following identity holds for certain special parameters of \( \pFq{2}{1}{\bullet,\bullet}{\bullet}{z} \), at \( z = -1 \),
		\begin{align}
		& \label{eqn:other} b \cdot \pFq{2}{1}{a,a+b}{a+1}{-1} + a \cdot  \pFq{2}{1}{b,a+b}{b+1}{-1} = \frac{\Gamma(1+a)\Gamma(1+b)}{\Gamma(a+b)} \,.
		\end{align}
	\end{Lem}
	
	\begin{proof}
		We apply the following transformation formula \cite[\S2.4]{Bailey} and integral representation \cite[\S1.5]{Bailey},
		\begin{align*}
		\pFq{2}{1}{a,b}{c}{-1} &= 2^{-a} \pFq{2}{1}{a,c-b}{c}{\tfrac{1}{2}} \\
		\pFq{2}{1}{a,b}{c}{z} &= \frac{\Gamma(c)}{\Gamma(b)\Gamma(c-b)} \int_0^1  t^{b-1}(1-t)^{c-b-1} (1 - t z)^{-a} \mathrm{d}t \,.
		\end{align*}
		We obtain
		\begin{align*}
		& b \cdot \pFq{2}{1}{a,a+b}{a+1}{-1} + a \cdot  \pFq{2}{1}{b,a+b}{b+1}{-1} \\
		& = 2^{-a} b \cdot  \pFq{2}{1}{1-b,a}{a+1}{\tfrac{1}{2}} + 2^{-b} a \cdot \pFq{2}{1}{1-a,b}{b+1}{\tfrac{1}{2}} \\
		& = 2^{-a} b \frac{\Gamma(a+1)}{\Gamma(a)\Gamma(1)} \int_0^1 t^{a-1} (1 - t/2)^{b-1} \mathrm{d}{t}  + 2^{-b} a \frac{\Gamma(b+1)}{\Gamma(b)\Gamma(1)} \int_0^1 t^{b-1} (1 - t/2)^{a-1} \mathrm{d}{t} 
		\end{align*}
		Simplify the Gamma function combinations, and set \( t = 2s \) in the first integral, and \( t = 2(1-s) \) in the second integral,
		\[
		=  ab  \int_0^{1/2} (s)^{a-1} (1 - s)^{b-1} \mathrm{d}{s} +  ab \int_1^{1/2} (1-s)^{b-1} s^{a-1} (-\mathrm{d}{s}) \,.
		\]
		The two integrals can now be combined (reverse the bounds of the second integral, to remove the minus sign in \( -\mathrm{d}s \)); the result evaluates as the beta function \( \beta(a,b) \coloneqq \frac{\Gamma(a)\Gamma(b)}{\Gamma(a+b)} \), giving
		\[
		= a b \int_0^1 s^{a-1}(1-s)^{b-1} \mathrm{d}s = \frac{a b \Gamma(a)\Gamma(b)}{\Gamma(a+b)} = \frac{\Gamma(a+1)\Gamma(b+1)}{\Gamma(a+b)} \,,
		\]
		as claimed.
	\end{proof}
	
	We now consider the depth \( {>}1 \) part of \autoref{conj:z111evenbar}, and show how to express its generating series via differential operators applied to the \( 1 - \pFq{2}{1}{x, -s}{1-s}{-1} \) generating series.  \Autoref{thm:main} will then give us an expression for this via \( \mathcal{D}_i \) applied to Gamma function combinations, from which \autoref{thm:z11evbar} follows.
	
	\begin{Prop}\label{prop:gs}
		The following generating series identity holds as a formal power-series in \( x \),
		\begin{equation}\label{eqn:gs:lhs}
		\begin{aligned}
		& \sum_{k=0}^\infty \sum_{m=1}^\infty \Big\{ \zeta(\{1\}^k, \overline{2m}) + \sum_{n=m+1}^{m+\lceil k/2 \rceil} [z^{2m-2}] E_{2n-3}(z) \zeta(\{1\}^{k+1-2(n-m)}, \overline{2n-1})	\Big\} x^{k+1} s^{2m-1} \\
		& = -g(x,-s) + \Big( \frac{2s}{1 + \exp(x \partial_b)} \Big\rvert_{b=s} \Big) \circ \Big( \frac{g(x,-b) + g(x,b)}{2b} \Big) \,,
		\end{aligned}
		\end{equation}
		where \( g(x,y) = 1 - \pFq{2}{1}{x,-y}{1-y}{-1} \) is the above generating series of \( \zeta(\{1\}^{k-1},\overline{n+1}) \).
		
		\begin{proof}
			We expand out the right hand side to check.  Firstly note
			\[
			\frac{g(x,-b) + g(x,b)}{2b} = \sum_{k,n>0} \zeta(\{1\}^{k-1}, \overline{2n+1}) x^k b^{2n-1} \,,
			\]
			Then by definition of the action of the differential operator (without the factor \( s \)), we find
			\begin{align*}
			\Big( \frac{2}{1 + \exp(x \partial_b)} \Big\rvert_{b=s} \Big) \circ \Big(\frac{b}{x}\Big)^{N} 
			&= \sum_{i=0}^N \frac{E_i(0)}{i!} x^i  \cdot \frac{\partial^i}{\partial b^i} \Big\rvert_{b=s} \Big(\frac{b}{x}\Big)^{N} \\
			&= \sum_{i=0}^N \frac{E_i(0)}{i!} \frac{N!}{(N-i)!} \Big(\frac{s}{x}\Big)^{N-i} \\
			&= E_N\Big(\frac{s}{x}\Big) \,
			\end{align*}
			via standard properties of the Euler polynomials.  So combining these observations (now with the factor \( s \) in the operator, which commutes with \( x \) and \( \partial_b \)), we have
			\begin{align}
			\label{eqn:diff_on_g}
			& \Big( \frac{2s}{1 + \exp(x \partial_b)} \Big\rvert_{b=s} \Big)  \circ \frac{g(x,-b) + g(x,b)}{2b}  
			=  \sum_{k,n>0} \zeta(\{1\}^{k-1}, \overline{2n+1}) x^{k+2n-1} s E_{2n-1}\Big(\frac{s}{x}\Big) 
			\end{align}
			By noting \( E_{2n-1}(z) \) has non-zero coefficient exactly for \( z^{2i} \), \( 0 \leq i \leq n-1 \), and for \( z^{2n-1} \) (whose coefficient is always 1), we can write
			\[
			E_{2n-1}\Big(\frac{s}{x}\Big) =  \sum_{m=1}^{n} \Big(\frac{s}{x}\Big)^{2m-2} [z^{2m-2}] E_{2n-1}(z) + \Big(\frac{s}{x}\Big)^{2n-1}
			\]
			Substitute this into the right-hand side of \autoref{eqn:diff_on_g} to obtain
			\begin{align*}
			\sum_{k,n>0} & \zeta(\{1\}^{k-1}, \overline{2n+1}) x^{k} s^{2n} \\[-1em]
			& \quad\quad + \sum_{k,n>0} \sum_{m=1}^{n} [z^{2m-2}] E_{2n-1}(z) \zeta(\{1\}^{k-1}, \overline{2n+1}) x^{k+2n-2m+1} s^{2m-1} .
			\end{align*}
			In the triple sum, now set \( k \mapsto k - 2n+2m \), then switch the \( n \) and \( m \) sums, and set \( n \mapsto n-m \)
			\begin{align*}
			\sum_{k,n>0} & \zeta(\{1\}^{k-1}, \overline{2n+1}) x^{k} s^{2n}  \\[-1ex]
			& \quad\quad + \sum_{m=1}^{\infty} \sum_{n=0}^\infty \sum_{k=2n+1}^{\infty} [z^{2m-2}] E_{2n+2m-1}(z) \zeta(\{1\}^{k-2n-1}, \overline{2n+2m+1}) x^{k+1} s^{2m-1} \,.
			\end{align*}
			Now in the triple sum, switch the \( n \) and \( k \) sums, and set \( n \mapsto n - m - 1 \),
			\begin{align*}
			\sum_{k,n>0} & \zeta(\{1\}^{k-1}, \overline{2n+1}) x^{k} s^{2n}  \\[-1em]
			& \quad\quad + \sum_{m=1}^{\infty} \sum_{k=1}^{\infty} \sum_{n=m+1}^{m+\lceil k/2 \rceil}  [z^{2m-2}] E_{2n - 3}(z) \zeta(\{1\}^{k+1-2(n-m)}, \overline{2n-1}) x^{k+1} s^{2m-1} \,.
			\end{align*}
			Then in the triple sum we can extend the sum over \( k \) to start from \( 0 \), as the \( n \) sum is empty in this case.  The triple sum is now exactly as appears in the conclusion.   Finally, it is immediate to see adding \( -g(x,-s) \) to both sides gives us \autoref{eqn:gs:lhs}.
		\end{proof}
	\end{Prop}
	
	We then obtain the following corollary of \autoref{thm:main}, by setting \( f(a,b) = \pFq{2}{1}{a,b}{1+b}{-1} \).
	
	\begin{Cor}\label{cor:z11evbar}
		The following identity holds as a formal power-series in \( x \),
		\begin{align*}
		& \sum_{k=0}^\infty \sum_{m=1}^\infty \Big\{ \zeta(\{1\}^k, \overline{2m}) + \sum_{n=m+1}^{m+\lceil k/2 \rceil} [z^{2m-2}] E_{2n-3}(z) \zeta(\{1\}^{k+1-2(n-m)}, \overline{2n-1})	\Big\} x^{k+1} y^{2m-1} \\
		& = \mathcal{D}_1 \circ \Big( \frac{\Gamma(1+a)\Gamma(1+b)}{\Gamma(a+b)}  \Big) + \mathcal{D}_2 \circ 1 +
		\underbrace{\mathcal{D}_3 \circ 0}_{=0}
		\end{align*}
		
		\begin{proof}
			Set \( f(a,b) = \pFq{2}{1}{a,b}{1+b}{-1} \) in \autoref{thm:main}.  Then with \(  g(x,y) = 1 - \pFq{2}{1}{x,-y}{1-y}{-1}  \), we see
			\begin{align*}
			\left\{ 
			\begin{aligned}
			& f(x,s) - f(0,s) = -g(x,-s) \,, \text{ and } \\
			& - \tfrac{ f(x, -b) + f(x, b)}{2b} + \tfrac{ f(0, -b) + f(0, b) }{2b} = \tfrac{g(x,-b) + g(x,b)}{2b} \,.
			\end{aligned}
			\right.
			\end{align*}
			This means the left-hand side of \autoref{thm:main} reduces to the generating series on the left-hand side of the statement of the corollary, by \autoref{eqn:gs:lhs} in \autoref{prop:gs}.
			
			On the other hand,
			\begin{align*}
			\left\{ 
			\begin{aligned}
			& \mathcal{D}_1 \circ \big(b f(a+b, a) + a f(a+b, b)\big) = \mathcal{D}_1 \circ \Big( \frac{\Gamma(1+a)\Gamma(1+b)}{\Gamma(a+b)} \Big)  \\
			& \mathcal{D}_2 \circ f(0, b) = \mathcal{D}_2 \circ 1 \\[1ex]
			& \mathcal{D}_3 \circ \big( \partial_b ( f(0, b) - f(0, -b) ) \big) = \mathcal{D}_3 \circ \big( \partial_b (1 - 1) \big) = \mathcal{D}_3 \circ 0
			\end{aligned} \right.
			\end{align*}
			by \autoref{eqn:other} for \( \mathcal{D}_1 \), and by the definition of \( {}_pF_q \) giving \( \pFq{2}{1}{0,\pm b}{1\pm b}{z} = 1 \) for \( \mathcal{D}_2, \mathcal{D}_3 \).  This has proven the corollary.
		\end{proof}
	\end{Cor}
	
	With this, \autoref{thm:z11evbar} from the introduction is proven to follow from \autoref{thm:main}, as is \autoref{conj:z111evenbar} (the original motivation for our investigation).  In the next section we complete the proof of \autoref{thm:main}, thus concluding the paper.
	
	\section{Proof of Main Theorem}\label{sec:proof}
	
	Introduce the following differential operators:
	\[
	\mathcal{L}_1 = \big( 1 + \exp(x \partial_s)\big) \, \frac{1}{2s} \Big\rvert_{s=b} \quad \,, \text{ and } \quad
	\mathcal{L}_2 = \frac{2s}{1 + \exp(x \partial_b)} \Big\rvert_{b=s} \,.
	\]
	These two differential operators are inverses (see \autoref{lem:inv} below), so by applying \( \mathcal{L}_1 \) to both sides of \autoref{thm:main}, we obtain the following identity which we will prove directly,
	\begin{equation}\label{eqn:goal}
	\begin{aligned}[c]
	&  \mathcal{L}_1 \circ ( f(x, s) - f(0, s)) \\
	& + \bigg( {-} \frac{ f(x, -b) + f(x, b)}{2b} + \frac{ f(0, -b) + f(0, b) }{2b} \bigg) \\[2ex]
	& \overset{?}{=} \begin{aligned}[t] 
	&  \mathcal{L}_1  \circ \mathcal{D}_1 \circ \big(b f(a+b, a) + a f(a+b, b)\big) \\
	& +  \mathcal{L}_1  \circ\mathcal{D}_2 \circ f(0, b) \quad + \quad  \mathcal{L}_1 \circ \mathcal{D}_3 \circ \big( \partial_b ( f(0, b) - f(0, -b) ) \big) \,. \end{aligned}
	\end{aligned}
	\end{equation}
	Then by applying the differential operator \( \mathcal{L}_2 \), which is inverse to \( \mathcal{L}_1 \) (see below), we will obtain the desired identity for \autoref{thm:main}. \medskip
	
	We start with some lemmas about the basic differential operators.
	\begin{Lem}\label{lem:trans}
		On functions \( \phi(y) \), the operator \( \exp(x \partial_y) \) acts as translation by \( x \), that is
		\[
		\exp(x \partial_y) \circ \phi(y) = \phi(x + y) \,.
		\]
		
		\begin{proof}
			This is a standard result about differential operators, and it follows directly via Taylor series,\vspace{-1em}
			\[
			\exp(x \partial_y) \circ \phi(y) = \sum_{i=0}^\infty \frac{x^i \phi^{(i)}(y)}{i!} = \phi(x + y) \,. \qedhere
			\]
		\end{proof}
	\end{Lem}
	
	\begin{Lem}\label{lem:inv}
		The differential operators
		\[
		\mathcal{L}_1 = \big( 1 + \exp(x \partial_s)\big) \, \frac{1}{2s} \Big\rvert_{s=b} \quad \,, \text{ and } \quad
		\mathcal{L}_2 = \frac{2s}{1 + \exp(x \partial_b)} \Big\rvert_{b=s} \,,
		\]
		are inverses, that is \( \mathcal{L}_1 \circ \mathcal{L}_2 = \mathcal{L}_2 \circ \mathcal{L}_1 = \operatorname{id} \).
		
		\begin{proof}
			This is a standard result in the theory of differential operators, and follows by straightforward computation using the definitions.  We briefly indicate how this works. \medskip
			
			\paragraph{\em Case \( \mathcal{L}_1 \circ \mathcal{L}_2 \):}  Unravelling the definitions, we have
			\begin{equation}
			\label{eqn:l1l2}
			\Big( \big( 1 + \exp(x \partial_s)\big) \, \frac{1}{2s} \Big\rvert_{s=b} \Big)  \circ \Big( \frac{2s}{1 + \exp(x \partial_b)} \Big\rvert_{b=s} \Big) \circ \phi(b) 
			= \sum_{i=0}^\infty \frac{E_i(0) x^i \big( \phi^{(i)}(b) + \phi^{(i)}(b+x) \big)}{2 \, i!} \,.
			\end{equation}
			By Taylor expanding \( \phi^{(i)}(b+x) \), we have
			\[
			\sum_{i=0}^\infty \frac{E_i(0) x^i \phi^{(i)}(b+x)}{2 \, i!} = \sum_{i=0}^\infty \sum_{m=0}^\infty \frac{E_i(0) x^i \cdot x^m \phi^{(i+m)}(b)}{2 \, m! \, i!} = \sum_{M=0}^\infty \sum_{i=0}^{M} \frac{E_i(0)}{i! (M-i!)} x^M \phi^{(M)}(b) \,,
			\]
			where \( M = m+i \) in the second equality.  The following translation property of Euler polynomials, \( E_M(x+y) = \sum_{i=0}^M E_i(x) \binom{M}{i} y^{M-i} \), shows that sum over \( i \) above is \( E_i(1) \).  So \autoref{eqn:l1l2} becomes
			\[
			\mathcal{L}_1 \circ \mathcal{L}_2 \circ \phi(b) =  \sum_{i=0}^\infty \frac{(E_i(0) + E_i(1))x^i \phi^{(i)}(b)}{2 \, i!} = \phi(b) \,;
			\]
			the identity \( E_i(x) + E_i(1+x) = 2x^i \) shows \( E_i(0) + E_i(1) = 2\delta_{i=0} \), so only the \( i = 0 \) term survives.\medskip
			
			\paragraph{\em Case \( \mathcal{L}_2 \circ \mathcal{L}_1 \)}  On the other hand, unravelling the definitions here gives
			\begin{equation}
			\label{eqn:l2l1}
			\begin{aligned}[c]
			\Big( \frac{2s}{1 + \exp(x \partial_b)} \Big\rvert_{b=s} \Big) \circ \Big( \big( 1 + \exp(x \partial_s)\big) \, \frac{1}{2s} \Big\rvert_{s=b} \Big) \circ \phi(s)  \\
			= \sum_{i=0}^\infty \frac{E_i(0) x^i}{2 \, i!} \, \partial^i_b \Big( \frac{\phi(b)}{b} + \frac{\phi(b+x)}{b+x} \Big) \Big\rvert_{b=s}  \,.
			\end{aligned}
			\end{equation}
			but with \( \vartheta(s) \coloneqq \frac{\phi(s)}{s} \), we have
			\[
			\mathcal{L}_2 \circ \mathcal{L}_1 \circ \phi(s) = s \sum_{i=0}^\infty \frac{E_i(0) x^i}{2 \, i!} \big( \vartheta^{(i)}(s) + \vartheta^{(i)}(s+x) \big) = s \, \vartheta(s) = \phi(s) \,,
			\] 
			by the calculation above.\medskip
			
			\noindent This has shown that \( \mathcal{L}_1 \) and \( \mathcal{L}_2 \) are inverse differential operators, so the result is proven.
		\end{proof}
	\end{Lem}
	
	We now collect a number of lemmas which allow us to easily compute the results of the translation operator \( \exp(x \partial_s) \) acting on differential operators which have been defined via Laplace transforms.
	
	\begin{Lem}\label{lem:laplace:stranslate}
		For \( f \in \mathbb{Q}[\partial_a,\partial_b,y]\llbracket x \rrbracket \), the following differential operators agree
		\[
		\exp(x \partial_s) \mathscr{L}_y\{f(y, \partial_a,\partial_b; x)\}(s) = \mathscr{L}_y\{ f(y,\partial_a,\partial_b; x) \exp(-x y) \}(s) \,.
		\]
		
		\begin{proof}
			Firstly, a standard property of Laplace transforms is 
			\[
			\partial_s^n \mathscr{L}_y\{g(y)\}(s) = \mathscr{L}_y\{ g(y) (-y)^n \}(s)  \,,
			\]
			which is immediate to verify.  Now we check that the coefficient of \( x^N \) on both sides agrees.  Write
			\[
			f(y, \partial_a, \partial_b; x) = \sum_{i=0}^\infty \lambda_i(y, \partial_a, \partial_b) x^i \,,
			\]
			then
			\begin{align*}
			[x^N] \exp(x \partial_s) \mathscr{L}_y\{f(y, \partial_a,\partial_b; x)\}(s) 
			& = \sum_{i+j=N} \frac{\partial_s^i}{i!} \mathscr{L}_y\big\{ \lambda_j(y, \partial_a, \partial_b)  \big\}(s) \\
			& =  \mathscr{L}_y\Big\{ \sum_{i+j=N} \lambda_j(y, \partial_a, \partial_b) \frac{(-y)^i}{i!} \Big\}(s) \\
			& = [x^N] \mathscr{L}_y\Big\{ f(y, \partial_a, \partial_b) \exp(-x y) \Big\}(s) \,.
			\end{align*}
			So the claim is proven.
		\end{proof}
	\end{Lem}
	
	\begin{Lem}\label{lem:op1}
		For \( f \in \mathbb{Q}[\partial_a,\partial_b,y]\llbracket x \rrbracket \), the following identity holds
		\begin{align*}
		& \exp(x \partial_s) \circ \mathscr{L}_y\big\{ f(y, \partial_a, \partial_b; x) \big\}(s) \big\rvert_{a=-s, b=s} \\
		& = \mathscr{L}_y\big\{ f(y, \partial_a, \partial_b; x) \exp(-x \partial_a + x \partial_b) \exp(-x y) \big\}(s) \big\rvert_{a=-s,b=s}
		\end{align*}
		
		\begin{proof}
			From \autoref{lem:laplace:stranslate}, we have
			\begin{align*}
			& \exp(x \partial_s) \circ \mathscr{L}_y\big\{ f(y, \partial_a, \partial_b; x) \big\}(s) \big\rvert_{a=-s, b=s} \\
			& = \mathscr{L}_y\big\{ f(y, \partial_a, \partial_b; x) \exp(-xy) \big\}(s) \big\rvert_{a=-s-x, b=s+x}
			\end{align*}
			We can then rewrite the substitution as follows: \( \phi(a,b) \big\rvert_{a=-s-x, b=s+x} = \phi(a-x,b+x) \big\rvert_{a=-s, b=s} \).  Then we can express this through a translation operator with respect to \( a \) and \( b \), namely
			\[
			\phi(a-x,b+x) = \exp(-x\partial_a + x\partial_b) \circ \phi(a,b) \,.
			\]
			So overall
			\begin{align*}
			& \mathscr{L}_y\big\{ f(y, \partial_a, \partial_b; x) \exp(-xy) \big\}(s) \big\rvert_{a=-s-x, b=s+x} \\
			& = \mathscr{L}_y\big\{ f(y, \partial_a, \partial_b; x) \exp(-xy) \big\}(s) \exp(-x\partial_a + x\partial_b) \big\rvert_{a=-s, b=s} \\
			& = \mathscr{L}_y\big\{ f(y, \partial_a, \partial_b; x) \exp(-xy) \exp(-x\partial_a + x\partial_b) \big\}(s) \big\rvert_{a=-s, b=s} \,,
			\end{align*}
			where the linearity in the last step is straightforward to check for formal power-series: it holds directly coefficient by coefficient.
		\end{proof}
	\end{Lem}
	The following two lemmas hold by analogous proofs to \autoref{lem:op1} above.
	\begin{Lem}\label{lem:op2}
		For \( f \in \mathbb{Q}[\partial_b,y]\llbracket x \rrbracket \), the following identity holds
		\begin{align*}
		& \exp(x \partial_s) \circ \mathscr{L}_y\{f(y, \partial_b; x)\}(s) \rvert_{b=s} = \mathscr{L}_y\{f(y,\partial_b;x) \exp(x \partial_b) \exp(-x y)\}(s)  \rvert_{b=s}
		\end{align*}
	\end{Lem}
	
	\begin{Lem}\label{lem:op3}
		For \( f \in \mathbb{Q}[y]\llbracket x \rrbracket \), the following identity holds
		\begin{align*}
		\exp(x \partial_s) \circ \mathscr{L}_y\{f(y;x)\}(s) \rvert_{b=s} = \mathscr{L}_y\{ f(y;x) \exp(x \partial_b)\exp(-xy) \}(s)  \rvert_{b=s} \,.
		\end{align*}
	\end{Lem}
	
	\paragraph{\bf The power-series \( \Psi \), \( \Omega \):} Introduce the formal power series
	\[
	\Psi(z) \coloneqq \sum_{k=0}^\infty E_k(0) z^k 
	\]
	It also will be convenient to write 
	\begin{equation}\label{eqn:omega}
	\Omega(z) \coloneqq \Psi(z) + z \Psi'(z)\,,
	\end{equation}
	in the calculations below.
	We note that
	\[
	s \cdot \mathscr{L}_y\Big\{\frac{2}{1 + \exp( x y)}\Big\}(s)  = s \cdot \mathscr{L}_y\Big\{ \sum_{i=0}^\infty \frac{E_i(0)}{i!} (x y)^i \Big\}(s) = \sum_{i=0}^\infty E_i(0) \Big(\frac{x}{s}\Big)^i = \Psi\Big(\frac{x}{s}\Big) \,.
	\]
	Using the property \( \mathscr{L}_y\{ f(y) \exp(xy) \}(s) = \mathscr{L}_y\{ f(y) \}(s-x) \), which can be checked by directly (cf. \autoref{lem:op1}), we obtain
	\begin{align*}
	\frac{1}{s} & = \mathscr{L}_y\Big\{ \frac{1 + \exp(x y)}{2} \cdot \frac{2}{1 + \exp(xy)} \Big\}(s) \\
	& = \frac12\mathscr{L}_y\Big\{  \frac{2}{1 + \exp(xy)} \Big\}(s)
	+\frac12\mathscr{L}_y\Big\{ \frac{2}{2 + \exp(xy)} \Big\}(s-x) \\
	& = \frac1{2s}\Psi\Big(\frac{x}{s} \Big)
	+ \frac1{2(s-x)}\Psi\Big(\frac{x}{s-x} \Big)
	\end{align*}
	Setting \( s \mapsto s + x \), solving for \( \Psi\big( \frac{x}{s+x} \big) \), then differentiating (with respect to \( x \) or \( s \)), and solving for \( \Psi'\big( \frac{x}{s+x} \big) \), we obtain
	\begin{equation}\label{eqn:phi:trans}
	\left\{ 
	\begin{aligned}[c]
	\Psi\Big( \frac{x}{x+s} \Big) &= 2 - \frac{x+s}{s} \Psi\Big( \frac{x}{s} \Big) \,, \\
	\Psi'\Big( \frac{x}{x+s} \Big) &= - \Big(\frac{x+s}{s}\Big)^2 \Psi\Big( \frac{x}{s} \Big) - \Big(\frac{x+s}{s}\Big)^3 \Psi'\Big( \frac{x}{s} \Big) \,.
	\end{aligned}
	\right.
	\end{equation}
	(These identities can also be checked directly, using properties of the Euler polynomials.) \medskip
	
	We are now in a position to directly check the identity in \autoref{eqn:goal}.\smallskip
	
	\paragraph{\bf Left-hand side:} With the understanding of \( \exp(x\partial_s) \big\rvert_{s=b} \) from \autoref{lem:trans} , the left hand side of \autoref{eqn:goal} evaluates to
	\begin{align*}
	& \frac{f(x, b) - f(0, b)}{2b} + \frac{f(x, b+x) - f(0, b+x)}{2(b+x)}  \\
	&  \quad + \Big( {-} \frac{ f(x, -b) + f(x, b)}{2b} + \frac{ f(0, -b) + f(0, b) }{2b}  \Big) \\[2ex]
	&= \frac{f(0,-b) - f(x,-b)}{2b} +  \frac{f(x, b+x) - f(0, b+x)}{2(b+x)} \,.
	\end{align*}
	We need to show the right-hand side simplifies to this as well. \medskip
	
	\paragraph{\bf Operator \( \mathcal{L}_1 \circ \mathcal{D}_1 \):} 
	We split \( \mathcal{D}_1 \) into simpler pieces.  Put
	\begin{align*}
	\mathcal{D}_1^{[0]} & =  \frac{s}{2} \cdot \mathscr{L}_y\big\{ y \big\}(s) \big\rvert_{b=s} = \frac{1}{2s} \Big\rvert_{a=-s,b=s} \,, \\
	\mathcal{D}_1^{[1]} & = - \frac{s}{2} \cdot \mathscr{L}_y\Big\{
	\begin{aligned}[t]
	&  \partial_y \Big( y (\partial_a - \partial_b) \tanh\Big( \frac{x y}{2} \Big) \Big)  
	\Big\}(s) \Big\rvert_{a=-s,b=s} \,,
	\end{aligned} \\
	\mathcal{D}_1^{[2]} &= s \cdot \mathscr{L}_y\Big\{
	\begin{aligned}[t]
	& \frac{1}{1 + e^{-x(\partial_a - \partial_b + y)}} \Big( \frac{-1 + e^{-x y}}{x} e^{x \partial_b} + \big(1 - e^{-x(\partial_a - \partial_b)} \big) \frac{y}{1 + e^{x y}}\Big)  
	\Big\}(s) \Big\rvert_{a=-s,b=s} \,;
	\end{aligned}
	\end{align*}
	one checks with some straightforward trigonometric calculations that \( \mathcal{D}_1 = \mathcal{D}_1^{[0]} +  \mathcal{D}_1^{[1]} + \mathcal{D}_1^{[2]} \).  (Note that \( \mathcal{D}_1^{[2]} \) here comes from lines 2 and 3 in \autoref{eqn:d1}.)\medskip
	
	\paragraph{\em Case \( \mathcal{D}_1^{[0]} \):} Directly we compute
	\begin{align*}
	& \mathcal{L}_1 \circ \mathcal{D}_1^{[0]} \circ \big( b f(a+b,a) + a f(a+b,b) \big) \\
	& = (1 + \exp(x \partial_s)) \frac{1}{2s} \Big\rvert_{s=b} \circ \Big( \frac{1}{2s} (b f(a+b,a) + a f(a+b, b) ) \Big) \Big\rvert_{a=-s,b=s} \\
	& = (1 + \exp(x \partial_s)) \frac{1}{2s} \Big\rvert_{s=b} \circ \frac{f(0,-s) -  f(0, s)}{2} \\
	& = \frac{f(0,-b) -  f(0, b)}{4b} + \frac{f(0,-b-x) -  f(0, b+x)}{4(b+x)} \,.
	\end{align*}
	\medskip
	
	\paragraph{\em Case \( \mathcal{D}_1^{[1]} \):}	Now recalling
	\[	
	\tanh\Big( \frac{xy}{2} \Big) = 1 - \sum_{i=0}^\infty \frac{E_i(0)}{i!} \big( xy \big)^i \,,
	\]
	we find
	\[
	\mathscr{L}_y\Big\{\tanh\Big( \frac{x y}{2} \Big)\Big\}(s) = \frac{1}{s}  \sum_{i=1}^\infty E_i(0) \Big(\frac{x}{s}\Big)^i = \frac{1}{s} \Big(1  - \Psi\Big( \frac{x}{s} \Big) \! \Big) \,.
	\]
	Moreover, using the standard Laplace transform properties \( \mathscr{L}_y\{y f(y)\}(s) = -\partial_s \mathscr{L}_y\{ f(y) \}(s) \), and \( \mathscr{L}_y\{ \partial_y f(y)\}(s) = s \mathscr{L}_y\{ f(y) \}(s)  - f(0^-) \) (which transfer to formal power-series), we find
	\[
	-\frac{s}{2} \cdot \mathscr{L}_y\Big\{ \partial_y \Big( y \tanh\Big( \frac{x y}{2} \Big) \! \Big)\Big\}(s) = \frac{1}{2} \Big( - 1 + \underbrace{\Psi\Big( \frac{x}{s} \Big) + \frac{x}{s} \Psi'\Big( \frac{x}{s} \Big)}_{\text{$=\Omega(\tfrac{x}{s})$ from Eq. \eqref{eqn:omega}}} \! \Big) \,.
	\] 
	Reinstating the additional factor \( \partial_a - \partial_b \), which is present in \( \mathcal{D}_1^{[1]} \), and with the usual notation \( f^{(i,j)}(a,b) = \partial_a^i \partial_b^j f(a,b) \), we find
	\begin{align*}
	& \mathcal{L}_1 \circ \mathcal{D}_1^{[1]} \circ \big( b f(a+b,a) + a f(a+b,b) \big) \\
	& = (1 + \exp(x \partial_s)) \frac{1}{2s} \Big\rvert_{s=b} \circ \begin{aligned}[t] \frac{1}{2} \Big( & - 1 + \Omega\Big( \frac{x}{s} \Big) \! \Big) \\[-0.5ex]
	& \cdot (\partial_a - \partial_b) \big(b f(a+b,a) + a f(a+b, b) \big)  \Big\rvert_{a=-s,b=s} \end{aligned}
	\\
	& = (1 + \exp(x \partial_s)) \frac{1}{2s} \Big\rvert_{s=b} \circ \begin{aligned}[t] \frac{1}{2} \Big( & - 1 + \Omega\Big( \frac{x}{s} \Big) \! \Big) \\[-0.5ex]
	& \cdot \big( -f(0,-s) + f(0,s) + s f^{(0,1)}(0,-s) + s f^{(0,1)}(0,s)  \big) \end{aligned}
	\end{align*}
	Using \autoref{eqn:phi:trans}, one checks that
	\begin{equation}\label{eqn:phi:comb}
	\begin{aligned}[c]
	& \exp(x \partial_s) \frac{1}{2s} \circ \Big( - 1 + \Omega\Big( \frac{x}{s} \Big) \! \Big)
	= \frac{1}{2(s+x)} \Big( 1 - \Big( \frac{s+x}{s} \Big)^2 \Omega\Big( \frac{s}{x} \Big) \Big) \,.
	\end{aligned}
	\end{equation}
	So we find
	\begin{align*}
	& \mathcal{L}_1 \circ \mathcal{D}_1^{[1]} \circ \big( b f(a+b,a) + a f(a+b,b) \big) \\
	& = \begin{aligned}[t] 
	& \frac{1}{4} \Big( - 1 +  \Omega\Big( \frac{x}{b} \Big) \! \Big) \cdot \bigg\{ \frac{-f(0,-b) + f(0,b)}{b} + f^{(0,1)}(0,-b) + f^{(0,1)}(0,b)  \bigg\}
	\\
	& \begin{aligned}[t] {} + \frac{1}{4} \Big( 1 -  {} & \Big( \frac{b+x}{b} \Big)^2 \Omega\Big( \frac{x}{b} \Big) \! \Big) \\
	& \cdot \bigg\{ \frac{-f(0,-b-x) + f(0,b+x)}{b+x} + f^{(0,1)}(0,-b-x) + f^{(0,1)}(0,b+x)  \bigg\} \,. \end{aligned}
	\end{aligned}
	\end{align*}
	\medskip
	
	\paragraph{\em Case \( \mathcal{D}_1^{[2]} \):}
	Finally, for \( \mathcal{D}_1 \), we note that by \autoref{lem:op1}, we have
	\begin{align*}
	& (1 + \exp(x\partial_s)) \frac{1}{2s} \Big\rvert_{s=b} \circ \mathcal{D}_1^{[2]} \\
	& = \frac{1}{2} \mathscr{L}_y\Big\{   \frac{-1 + e^{-x y}}{x} e^{x \partial_b} + \big(1 - e^{-x(\partial_a - \partial_b)} \big)  \frac{y}{1 + e^{x y}}\Big)   \Big\}(s) \Big\rvert_{a=-b,s=b}  \\
	& = \frac{1}{2x} \Big( \frac{1}{s+x} - \frac{1}{s} \Big) e^{x \partial_b}  + \frac{1}{4s^2} \big(1 - e^{-x(\partial_a - \partial_b)} \big) \Omega\Big(\frac{x}{s}\Big) \Big\rvert_{a=-b,s=b} \,.
	\end{align*}
	So likewise, we have
	\begin{align*}
	& (1 + \exp(x\partial_s)) \frac{1}{2s} \Big\rvert_{s=b} \circ \mathcal{D}_1^{[2]} \circ \big(b f(a+b, a) + a f(a+b, b) \big) \\
	& = \begin{aligned}[t] 
	& \frac{1}{2x} \Big( \frac{1}{b+x} - \frac{1}{b} \Big) \big((b+x) f(x, -b) - b f(x, b+x) \big) \\
	& + \frac{1}{4b} \Omega\Big( \frac{x}{b} \Big) \big( f(0, -b) -  f(0, b) \big) 
	- \frac{b+x}{4b^2}  \Omega\Big( \frac{x}{b} \Big)
	\big(f(0, -b-x) - f(0, b+x) \big) 
	\end{aligned}
	\end{align*}
	We note already here that the \( \Omega \) terms in \( \mathcal{L}_1 \circ \mathcal{D}_1^{[2]} \) cancel with the \( \Omega(\bullet) f(\bullet)  \) terms in \( \mathcal{L}_1 \circ \mathcal{D}_1^{[1]} \).  (The remaining \( \Omega(\bullet) f^{(0,1)}(\bullet)  \) terms in \( \mathcal{L}_1 \circ \mathcal{D}_1^{[1]} \) will cancel with \( \mathcal{L}_1 \circ \mathcal{D}_3 \) below.) \medskip
	
	\paragraph{\bf Operator \( \mathcal{L}_1 \circ \mathcal{D}_2 \):} From \autoref{lem:op2}, we have that
	\begin{align*}
	(1 + \exp(x \partial_s)) \frac{1}{2s} \Big\rvert_{s=b} \circ \mathcal{D}_2 
	& = \frac{1}{2} \mathscr{L}_y\{ 1 - \exp(x(\partial_b - y))\} \Big\rvert_{a=-b,s=b} \\
	& = \frac{1}{2} \Big( \frac{1}{s} - \frac{1}{s+x} \exp(x \partial_b) \Big) \Big\rvert_{a=-b,s=b}
	\end{align*}
	So immediately, we find
	\[
	\mathcal{L}_1 \circ \mathcal{D}_2 \circ f(0,b) = \frac{f(0,b)}{2b} - \frac{f(0, b+x)}{2(b+x)} \,.
	\]
	
	\paragraph{\bf Operator \( \mathcal{L}_1 \circ \mathcal{D}_3 \):} Finally, we note that \( \mathcal{D}_3 \) can be rewritten as:
	\begin{align*}
	\mathcal{D}_3 &= s \cdot \mathscr{L}_y\Big\{ \partial_y \partial_x \, \frac{1}{2} x \tanh\Big(\frac{xy}{2}\Big)  \Big\}(s) \Big\rvert_{b=s}  \\
	& = s^2 \partial_x \frac{x}{2s} \Big( 1 - \Psi\Big( \frac{x}{s} \Big) \! \Big)  \Big\rvert_{b=s} \\
	& = \frac{s}{2} \Big( 1 - \Omega\Big( \frac{x}{s} \Big) \! \Big)  \Big\rvert_{b=s} \,.
	\end{align*}
	So we find (using again \autoref{eqn:phi:comb}),
	\begin{align*}
	&\mathcal{L}_1 \circ \mathcal{D}_3 \circ \partial_b\big( f(0, b) - f(0,-b) \big) \\
	& =  
	(1 + \exp(x\partial_s))\Big\rvert_{s=b} \circ 
	\frac{1}{4} \Big( 1 - \Omega\Big( \frac{x}{s} \Big)  \! \Big) \big( f^{(0,1)}(0, s) - f^{(0,1)}(0,-s) \big) 	\\[1ex]
	& =  
	\begin{aligned}[t]
	& \frac{1}{4} \Big( 1 - \Omega\Big( \frac{x}{b} \Big) \Big) \big\{ f^{(0,1)}(0, b) - f^{(0,1)}(0,-b) \big\} \\
	& + \frac{1}{4} \Big( -1 + \Big( \frac{b+x}{b} \Big)^2 \Omega\Big( \frac{x}{b} \Big) \! \Big) \big\{ f^{(0,1)}(0, b+x) - f^{(0,1)}(0,-x-b) \big\} \,. \end{aligned}
	\end{align*}
	We note that \( \mathcal{L}_1 \circ \mathcal{D}_3 \) cancels completely with the \( \Omega(\bullet) f^{(0,1)}(\bullet) \) terms in \( \mathcal{L}_1 \circ \mathcal{D}^{[1]}_1 \). \medskip
	
	\paragraph{\bf Right-hand side:}  Now sum all of the contributions above to \( \mathcal{L}_1 \circ \mathcal{D}_i \).  Recall that the \( \Omega \) terms in \( \mathcal{L}_1 \circ \mathcal{D}_1^{[2]} \) cancel with the \( \Omega(\bullet) f(\bullet)  \) terms in \( \mathcal{L}_1 \circ \mathcal{D}_1^{[1]} \), and that \( \mathcal{L}_1 \circ \mathcal{D}_3 \) cancels completely with the \( \Omega(\bullet) f^{(0,1)}(\bullet) \) terms in \( \mathcal{L}_1 \circ \mathcal{D}^{[1]}_1 \).  We are then left with the right-hand side being
	\begin{align}
	\tag{$\mathcal{D}_1^{[0]}$} 
	& \frac{f(0,-b) -  f(0, b)}{4b} + \frac{f(0,-b-x) -  f(0, b+x)}{4(b+x)}  \\
	\tag{$\mathcal{D}_1^{[1]}$} 
	& + \frac{1}{4} \big( - 1 \big) \cdot \bigg\{ \frac{-f(0,-b) + f(0,b)}{b}  \bigg\}  + \frac{1}{4} \big( 1  \big)  \cdot \bigg\{ \frac{-f(0,-b-x) + f(0,b+x)}{b+x} \bigg\} \\
	\tag{$\mathcal{D}_1^{[2]}$}
	& + \frac{1}{2x} \Big( \frac{1}{b+x} - \frac{1}{b} \Big) \big((b+x) f(x, -b) - b f(x, b+x) \big) \\
	\tag{$\mathcal{D}_2$}
	& + \frac{f(0,b)}{2b} - \frac{f(0, b+x)}{2(b+x)} 
	\end{align}
	This is readily seen to simplify (the last terms in lines 1 and 2 cancel, the first terms in lines 1 and 2 combine).  Overall we obtain
	\[
	= \frac{f(0,-b) - f(x,-b)}{2b} +  \frac{f(x, b+x) - f(0, b+x)}{2(b+x)} \,,
	\]
	which agrees with the left-hand side of \autoref{eqn:goal}.  So \autoref{eqn:goal} is proven, and by applying \( \mathcal{L}_2 \) to both sides, the main theorem is finally proven.

\end{document}